\documentclass[10pt]{article}
\usepackage{graphicx,float,latexsym,amsmath,amsfonts,amssymb,subeqnarray,colordvi,epstopdf}
\usepackage[latin1]{inputenc}
\usepackage{a4wide}

\def\R{\mathbb{R}}

\newtheorem{theorem}{Theorem}

\newtheorem{lemma}[theorem]{Lemma}

\newtheorem{assumption}[theorem]{Assumption}

\newenvironment{proof}{\noindent {\it Proof}~}{}

\newcommand{\rT}{{\rm{T}}}

\providecommand*{\diff}{\@ifnextchar^{\DIfF}{\DIfF^{}}}
\def\DIfF^#1{\operatorname{d}\nolimits^{#1}\gobblespace}
\def\gobblespace{\futurelet\diffarg\opspace}
\def\opspace{\let\DiffSpace\!
	\ifx\diffarg(\let\DiffSpace\relax
	\else
	\ifx\diffarg[\let\DiffSpace\relax
	\else
	\ifx\diffarg\{\let\DiffSpace\relax
	\fi\fi\fi\DiffSpace}

\providecommand*{\pderiv}[3][]{\frac{\partial^{#1}#2}{{\partial #3}^{#1}}}

\title{Numerical valuation of Bermudan basket options\\ 
via partial differential equations}
\author{Karel~J.~in 't Hout\footnote{Department of Mathematics and Computer Science,
University of Antwerp, Middelheimlaan 1, B-2020 Antwerp, Belgium.
\mbox{Email}: \texttt{\{karel.inthout,jacob.snoeijer\}@uantwerpen.be}.}
~and Jacob Snoeijer\footnotemark[\value{footnote}]
}
\date{\today}

\begin{document}

\maketitle

\begin{abstract}
\noindent
We study the principal component analysis (PCA) based approach introduced by Reisinger \& Wittum \cite{RW07} 
for the approximation of Bermudan basket option values via partial differential equations (PDEs).
This highly efficient approximation approach requires the solution of only a limited number of low-dimensional 
PDEs complemented with optimal exercise conditions.
It is demonstrated by ample numerical experiments that a common discretization of the pertinent PDE problems 
yields a second-order convergence behaviour in space and time, which is as desired.
It is also found that this behaviour can be somewhat irregular, and insight into this phenomenon is obtained.
\end{abstract}
\vspace{0.2cm}\noindent
\vspace{3mm}
\normalsize


\setcounter{equation}{0}
\section{Introduction}\label{sec:introduction}
This paper deals with the valuation of Bermudan basket options.
Basket options have a payoff depending on a weighted average of different assets.
Semi-closed analytic valuation formulas are generally lacking in the literature for these options.
Consequently, research into efficient and stable methods for approximating their fair values is of much interest.
The valuation of basket options gives rise to multidimensional time-dependent partial differential equations.
Here the spatial dimension $d\ge 2$ equals the number of different assets in the basket.
When $d$ is large, it is well-known that this leads to a highly challenging numerical task.
In the present paper we shall consider Bermudan-style basket options and investigate a principal component analysis based approach 
introduced by Reisinger \& Wittum \cite{RW07} and subsequently studied in e.g.~\cite{RW15,RW17,RW18} that renders this task feasible.

A \emph{European-style basket option} is a financial contract that provides the holder the right, but not the obligation, to buy or sell a 
given weighted average of $d$ assets at a specified future date $T$ for a specified price $K$.
Parameter $T$ is called the maturity time and $K$ the strike price of the option.
In this paper we assume the well-known Black--Scholes model.
Then the asset prices $S_\tau^i$ ($i=1,2,\ldots,d$) follow a multidimensional geometric Brownian motion, under the risk-neutral measure, 
given by the system of stochastic differential equations
\begin{equation*}\label{eq:SDE}
dS_\tau^i = r S_\tau^i d\tau + \sigma_i S_\tau^i dW_\tau^i \quad (0< \tau \le T,~1\le i\le d).
\end{equation*}
Here $\tau$ is time, with $0$ being the time of inception of the option, $r\ge 0$ is the given risk-free interest rate, $\sigma_i>0$ 
($i=1,2,\ldots,d$) are the given volatilities and $W_\tau^i$ ($i=1,2,\ldots,d$) is a multidimensional standard Brownian motion with given 
correlation matrix $(\rho_{ij})_{i,j=1}^d$.
Further, initial asset prices $S_0^i$ ($i=1,2,\ldots,d$) are given.
Let $u(s,t) = u(s_1,s_2,\ldots,s_d,t)$ be the fair value of a European basket option if at time $\tau = T-t$ the $i$-th asset price 
equals $s_i$ ($i=1,2,\ldots,d$), where $t$ is the time remaining till maturity of the option.
From financial mathematics theory it follows that $u$ satisfies the $d$-dimensional time-dependent partial differential equation (PDE)
\begin{equation}\label{eq:PDE_s}
\pderiv{u}{t}(s,t)
= \frac{1}{2} \sum_{i=1}^d \sum_{j=1}^d \sigma_i \sigma_j \rho_{ij} s_i s_j \frac{\partial^2 u}{\partial s_i \partial s_j}(s,t)
+ \sum_{i=1}^d r s_i \pderiv{u}{s_i}(s,t) - r u(s,t)
\end{equation}
for $(s, t) \in (0,\infty)^d \times (0,T]$.
PDE \eqref{eq:PDE_s} is also fulfilled whenever $s_i=0$ for any given $i$, defining a natural boundary condition.
In almost all financial applications, the correlation matrix has nonzero off-diagonal entries, and hence, \eqref{eq:PDE_s} contains mixed 
spatial derivative terms.
At maturity time of the option its fair value is known and determined by the particular option contract.
This yields the initial condition
\begin{equation}\label{eq:IC_s}
u(s,0) = \phi(s)
\end{equation}
for $s \in (0,\infty)^d$.
Here function $\phi$ is the given payoff of the option.

A \emph{Bermudan-style basket option} is a financial contract that provides the holder the right to buy or sell a given weighted average of 
$d$ assets for a specified price $K$ at one from a specified finite set of exercise times $\tau_1<\tau_2<\cdots<\tau_E=T$ with $\tau_1>0$.
Let $\alpha_{e} = T - \tau_{E-e}$ for $e=0,1,\ldots,E-1$ and $\alpha_{E} = T$.
Then the fair value function $u$ of a Bermudan basket option satisfies the PDE \eqref{eq:PDE_s}, with the natural boundary condition, on 
each time interval $(\alpha_{e-1},\alpha_e)$ for $e=1,2,\ldots,E$.
Next, the initial condition \eqref{eq:IC_s} holds and for $e=1,2,\ldots,E-1$ one has
\begin{equation}\label{eq:ICe_s}
u(s,\alpha_e) = \max \left( \phi(s) \,,\, \lim_{t\uparrow \alpha_e} u(s,t) \right)
\end{equation}
whenever $s \in (0,\infty)^d$.
Condition \eqref{eq:ICe_s} stems from the early exercise feature of Bermudan options and represents the optimal exercise condition. 
Notice that it is nonlinear.
In the present paper we shall consider the class of Bermudan basket put options.
These have a payoff function of the form
\begin{equation}\label{eq:payoff}
\phi(s) = \max \displaystyle \left(K - \sum_{i=1}^d \omega_i s_i\,,\, 0\right)
\end{equation}
with given fixed weights $\omega_i>0$ such that $\sum_{i=1}^d \omega_i = 1$.

An outline of the rest of this paper is as follows.

Following Reisinger \& Wittum \cite{RW07}, we first apply in Section \ref{sec:transformation} a useful coordinate transformation to
\eqref{eq:PDE_s} by using a spectral decomposition of the pertinent covariance matrix. 
This leads to a $d$-dimensional time-dependent PDE for a transformed option value function $w$ in which each coefficient is directly 
proportional to one of the eigenvalues.
In Section \ref{sec:approximation} this feature is employed to define a principal component analysis (PCA) based approximation 
${\widetilde w}$ to $w$. 
The key property of ${\widetilde w}$ is that it is defined by only a limited number of one- and two-dimensional PDEs.
In Section \ref{sec:exercise condition} a note on the optimal exercise condition is given.
Section \ref{sec:discretization} describes a common discretization of the one- and two-dimensional PDE problems by 
means of finite differences on a suitable nonuniform spatial grid followed by the Brian and Douglas ADI scheme on a uniform temporal 
grid.
In view of the nonsmoothness of the payoff function, cell averaging and backward Euler damping are applied.

The main contribution of our paper is given in Section \ref{sec:numericalexperiments}. 
Extensive numerical experiments are presented where we study in detail the error of the discretization described in Section 
\ref{sec:discretization} of the PCA-based approximation ${\widetilde w}$ defined in Section \ref{sec:approximation} for Bermudan 
basket options. 
Three financial parameter sets from the literature are considered, with number of assets $d=5,10,15$.
A second-order convergence behaviour is observed, which is as desired. 
It is also found that this behaviour can be somewhat irregular.
Additional numerical experiments are performed that yield insight into this phenomenon.

Section \ref{sec:conclusion} contains our conclusions and outlook.


\setcounter{equation}{0}
\section{Approximation approach}\label{}

\subsection{Coordinate transformation}\label{sec:transformation}
In this section the PDE \eqref{eq:PDE_s} for a Bermudan basket option is converted into a form that is the starting point for the solution 
approach discussed in the subsequent sections.
In the following, the elementary functions $\ln$, $\exp$, $\tan$ and $\arctan$ are to be taken componentwise whenever they are applied to 
vectors.

Consider the covariance matrix $\Sigma = \left(\Sigma_{ij} \right) \in \R^{d \times d}$ given by $\Sigma_{ij} = \sigma_i \rho_{ij} \sigma_j$ 
for $i,j = 1,2,\ldots,d$.
Let $Q \in \R^{d\times d}$ be an orthogonal matrix of eigenvectors of 
$\Sigma$ and $\Lambda = {\rm diag} (\lambda_1,\lambda_2,\ldots,\lambda_d) \in \R^{d\times d}$
a diagonal matrix of eigenvalues of $\Sigma$ such that $\Sigma = Q \Lambda Q^\rT$.
As in \cite{RW07}, consider the coordinate transformation
\begin{equation}\label{eq:transformationS2X}
x(s, t) = Q^\rT \left( \ln (s/K) - b(t) \right),
\end{equation}
with $b(t) = (b_1(t),b_2(t),\ldots,b_d(t))^\rT$ and $b_i(t) = (\tfrac{1}{2}\sigma_i^2-r)t$ for $1\le i\le d$.
Define the function $v$ by
\begin{equation*}
u(s,t) = v(x(s,t),t).
\end{equation*}
A straightforward calculation shows that $v$ satisfies
\begin{equation}\label{eq:PDE_x}
\pderiv{v}{t}(x,t) = \frac{1}{2} \sum_{k=1}^d \lambda_k  \frac{\partial^2 v}{\partial x_k^2}(x,t) - r v(x,t)
\end{equation}
whenever $x \in \R^d$, $t\in (\alpha_{e-1},\alpha_e)$, $1\le e \le E$.
The PDE \eqref{eq:PDE_x} is a pure diffusion equation, without mixed derivatives, and with a simple 
reaction term.

It is convenient to perform a second coordinate transformation \cite{RW07}, which maps the spatial
domain $\R^d$ onto the $d$-dimensional open unit cube,
\begin{equation}\label{eq:transformationX2Y}
y(x) = \frac{1}{\pi} \arctan (x) + \frac{1}{2}.
\end{equation}
Define the function $w$ by
\begin{equation*}
v(x,t) = w(y(x),t).
\end{equation*}
Then it is readily verified that 
\begin{equation}\label{eq:PDE_y}
\pderiv{w}{t}(y,t) =
\sum_{k=1}^d \lambda_k  \left[ p(y_k) \frac{\partial^2 w}{\partial y_k^2}(y,t)
+ q(y_k) \frac{\partial w}{\partial y_k}(y,t) \right] - r w(y,t) 
\end{equation}
whenever $y \in (0,1)^d$, $t\in (\alpha_{e-1},\alpha_e)$, $1\le e \le E$ with
\begin{equation*}
p(\eta) = \frac{1}{2\pi^2} \sin^4\! \left( \pi \eta \right),\quad
q(\eta) = \frac{1}{\pi} \sin^3\! \left( \pi \eta \right) \cos \left( \pi \eta \right)\quad
{\rm for}~\eta\in \R.
\end{equation*}
Clearly, the PDE \eqref{eq:PDE_y} is a convection-diffusion-reaction equation without mixed derivative terms.
Define the function $\psi$ by
\begin{equation}\label{eq:psi}
\psi (y,t) = \phi \left( K \exp \left[ Qx +b(t) \right] \right) ~~{\rm with}~~  x = \tan \left[ \pi(y-\tfrac{1}{2}) \right]
\end{equation}
whenever $y \in (0,1)^d$, $t\in [0,T]$.
Then for \eqref{eq:PDE_y} one has the initial condition
\begin{equation}\label{eq:IC_y}
w(y,0) = \psi (y,0) 
\end{equation}
together with the optimal exercise condition
\begin{equation}\label{eq:ICe_y}
w(y,\alpha_e) = \max \left(  \psi (y,\alpha_e) \,,\, \lim_{t\uparrow \alpha_e} w(y,t) \right) 
\end{equation}
for $y \in (0,1)^d$ and $e=1,2,\ldots,E-1$.
At the boundary $\partial D$ of the spatial domain $D=(0,1)^d$ we shall consider a Dirichlet condition.
In Appendix~\ref{sec:BC} the details of its derivation are provided, where the minor Assumption \ref{MatrixQ} 
on the matrix $Q$ is made.
For any given $k \in \{1,2,\ldots,d\}$ such that the entries of the $k$-th column of $Q$ are all strictly positive 
there holds
\begin{equation}\label{eq:BC_y}
w(y,t) = K e^{-r(t-\alpha_{e-1})}
\end{equation}
whenever $y \in \partial D$ with $y_k = 0$ and $t\in (\alpha_{e-1},\alpha_e)$, $1\le e \le E$.
On the complementary part of $\partial D$ a homogeneous Dirichlet condition is valid.

\subsection{Principal component analysis based approximation}\label{sec:approximation}
Let the eigenvalues of $\Sigma$ be ordered such that $\lambda_1 \geq \lambda_2 \geq \cdots \geq \lambda_d \geq 0$. 
In financial applications it often holds that $\lambda_1$ is significantly larger than the other eigenvalues.
Motivated by this observation, Reisinger \& Wittum \cite{RW07} proposed a principal component analysis (PCA) based approximation 
of the exact solution $w$ to the multidimensional PDE \eqref{eq:PDE_y}.
To this purpose, consider $w$ also as a function of the eigenvalues and write $w(y,t;\lambda)$ with 
$\lambda = (\lambda_1,\lambda_2,\ldots,\lambda_d)^\rT$.
Set
\begin{equation*}
\widehat{\lambda} = (\lambda_1,0,\ldots,0)^\rT ~~ {\rm and} ~~ \delta \lambda = \lambda - \widehat{\lambda} = (0,\lambda_2,\ldots,\lambda_d)^\rT.
\end{equation*}
Assuming sufficient smoothness of $w$, a first-order Taylor expansion at $\widehat{\lambda}$ yields
\begin{equation}\label{eq:TaylorExpansion}
w(y, t; \lambda) \approx w(y, t; \widehat{\lambda}) + \sum_{l=2}^d \delta \lambda_{l}\, \pderiv{w}{\lambda_{l}}(y, t; \widehat{\lambda}).
\end{equation}
The partial derivative ${\partial w}/{\partial \lambda_{l}}$ (for $2\le l \le d$) can be approximated by a forward finite difference,
\begin{equation}\label{eq:FiniteDifference}
\pderiv{w}{\lambda_{l}}(y, t; \widehat{\lambda}) \approx \frac{w(y, t; \widehat{\lambda} + \delta \lambda_l\, e_l) - w(y, t; \widehat{\lambda})}{\delta \lambda_l}\,,
\end{equation}
where $e_l$ denotes the $l$-th standard basis vector in $\R^d$.
Combining \eqref{eq:TaylorExpansion} and \eqref{eq:FiniteDifference}, gives
\begin{equation*}\label{eq:solutionApprox1}
w(y, t; \lambda) \approx w(y, t; \widehat{\lambda}) + \sum_{l=2}^d \left[ w(y, t; \widehat{\lambda} + \delta \lambda_{l}\, e_{l}) - w(y, t; \widehat{\lambda}) \right].
\end{equation*}
Write 
\begin{equation*}
w^{(1)}(y, t) = w(y, t; \widehat{\lambda}) ~~{\rm and}~~ w^{(1,\, l)}(y, t) = w(y, t; \widehat{\lambda} + \delta \lambda_{l}\, e_{l}). 
\end{equation*}
Then the PCA-based approximation reads
\begin{equation}\label{eq:solutionApprox2}
w(y, t) \approx {\widetilde w}(y, t) = w^{(1)}(y, t) + \sum_{l=2}^d \left[ w^{(1,\, l)}(y, t) - w^{(1)}(y, t) \right]
\end{equation}
whenever $y \in (0,1)^d$, $t\in (\alpha_{e-1},\alpha_e)$, $1\le e \le E$.
By definition, $w^{(1)}$ satisfies the PDE \eqref{eq:PDE_y} with $\lambda_k$ being set to zero for all $k\not= 1$ and $w^{(1,\, l)}$ satisfies 
\eqref{eq:PDE_y} with $\lambda_k$ being set to zero for all $k\not\in \{1,l\}$,
which is completed by the same initial condition, optimal exercise condition and boundary condition as for $w$, discussed above.

In financial applications one is often interested in the option value at inception in the single point $s=S_0$, where $S_0 = (S_0^1, S_0^2,\ldots, S_0^d)^\rT$ 
is the vector of known asset prices.
Let 
\begin{equation*}
Y_0 = y(x(S_0,T))\in (0,1)^d 
\end{equation*}
denote the corresponding point in the $y$-domain.
Then $w^{(1)}(Y_0,T)$ can be obtained by solving a one-dimensional PDE on the line segment $L_1$ in the $y$-domain that is parallel to the 
$y_1$-axis and passes through $y=Y_0$.
In other words, $y_k$ can be fixed at the value $Y_{0,k}$ whenever $k\not= 1$.
Next, $w^{(1,\, l)}(Y_0,\! T)$ (for $2\le l \le d$) can be obtained by solving a two-dimensional PDE on the plane segment $P_l$ in the $y$-domain 
that is parallel to the $(y_1,y_l)$-plane and passes through $y=Y_0$.
Thus, in this case, $y_k$ can be fixed at the value $Y_{0,k}$ whenever $k\not\in \{1,l\}$.

In view of the above key observation, computing the PCA-based approximation \eqref{eq:solutionApprox2} for $(y, t)=(Y_0,T)$ requires solving 
just 1 one-dimensional PDE and $d-1$ two-dimensional PDEs. 
This clearly yields a main computational advantage compared to solving the full $d$-dimensional PDE whenever $d$ is large. 
Moreover, the different terms in \eqref{eq:solutionApprox2} can be computed in parallel.

Reisinger \& Wissmann \cite{RW17} have given a rigorous analysis of the error in the PCA-based approximation relevant to European basket 
options.
Under a mild assumption on the payoff function $\phi$, they proved that ${\widetilde w} - w = {\cal O}\left( \lambda_2^2 \right)$ in the 
maximum norm.

\subsection{A note regarding the optimal exercise condition}\label{sec:exercise condition}
Let $1\le e\le E-1$ and write $\psi_e(y)=\psi (y,\alpha_e)$.
Let $y\in L_1$, which forms the intersection of $L_1$ and $P_2,\ldots,P_d$.
By the optimal exercise condition~\eqref{eq:ICe_y}, the natural approximation to $w(y,t)$ at $t=\alpha_e$ based on ${\widetilde w}$ is  
\begin{eqnarray*}
w(y,\alpha_e) 
&\approx& \max \Big( \psi_e(y) \,,\,  \displaystyle\lim_{t\uparrow \alpha_e} {\widetilde w}(y,t) \Big)\\
&=& \displaystyle\lim_{t\uparrow \alpha_e} \max \Big( \psi_e(y) \,,\,  {\widetilde w}(y,t) \Big)\\
&=& \displaystyle\lim_{t\uparrow \alpha_e} \max \Big(\psi_e(y) \,,\, w^{(1)}(y, t) + \sum_{l=2}^d \left[ w^{(1,\, l)}(y, t) - w^{(1)}(y, t) \right] \Big).
\end{eqnarray*}
On the other hand, by construction of $w^{(1)}$ and $w^{(1,\, l)}$ ($2\le l\le d$), we have 
\begin{eqnarray*}
w(y,\alpha_e) 
&\approx& {\widetilde w}(y, \alpha_e)\\
&=& w^{(1)}(y,\alpha_e) + \sum_{l=2}^d \left[ w^{(1,\, l)}(y,\alpha_e) - w^{(1)}(y,\alpha_e) \right]\\
&=& \displaystyle\lim_{t\uparrow \alpha_e} \left( \max \Big(\psi_e(y) \,,\, w^{(1)}(y, t) \Big) + 
\sum_{l=2}^d \left[ \max \Big(\psi_e(y) \,,\, w^{(1,\, l)}(y, t)\Big) - \max \Big(\psi_e(y) \,,\, w^{(1)}(y, t) \Big) \right] \right).
\end{eqnarray*}
It may hold that
\begin{equation*}
{\widetilde w}(y, \alpha_e) \not=
\max \Big( \psi_e(y) \,,\,  \displaystyle\lim_{t\uparrow \alpha_e} {\widetilde w}(y,t) \Big),
\end{equation*}
and hence, the PCA-based approximation ${\widetilde w}$ does not satisfy the optimal exercise condition.
A further investigation into this will be the subject of future research.

\subsection{Discretization}\label{sec:discretization}
To arrive at the values $w^{(1)}(Y_0,T)$ and $w^{(1,\, l)}(Y_0,T)$ (for $2\le l \le d$) in the approximation ${\widetilde w}(Y_0,T)$ of 
$w(Y_0,T)$ we perform finite difference discretization of the pertinent one- and two-dimensional PDEs on a (Cartesian) nonuniform spatial 
grid, followed by a suitable implicit time discretization.
Let $\kappa_0 = \tfrac{1}{2}$ and $\kappa_1 >0$.
Note that the point $(\kappa_0,\kappa_0,\ldots,\kappa_0)^\rT$ in the $y$-domain corresponds to the point $(K,K,\ldots,K)^\rT$ in the 
$s$-domain if $t=0$.
For any given $k\in \{1,2,\ldots,d\}$ a nonuniform mesh $0 = y_{k,0} < y_{k,1} < \ldots < y_{k,m+1}=1$ in the $k$-th spatial direction is 
defined by (see e.g.~\cite{H17})
\begin{equation*}
y_{k,i} =\varphi(\xi_i) ~~{\rm with}~~ \xi_i=\xi_{\rm min}+i\Delta \xi,~ \Delta \xi =\frac{\xi_{\rm max}-\xi_{\rm min}}{m+1} \quad 
(i=0,1,\ldots,m+1),
\end{equation*}
with
\begin{equation*}
\varphi(\xi) = \kappa_0 + \kappa_1 \sinh(\xi)\quad (\xi_{\min}\le \xi \le \xi_{\max})
\end{equation*}
and
\[
\xi_{\min} = -\sinh^{-1}(\kappa_0/\kappa_1) ~~{\rm and}~~ \xi_{\max} = \sinh^{-1}((1-\kappa_0)/\kappa_1).
\]
Remark that $\xi_{\max}  = - \xi_{\min}$ since $\kappa_0 = \tfrac{1}{2}$. 
The parameter $\kappa_1$ controls the fraction of mesh points that lie in the neighborhood of $\kappa_0$. 
We make the heuristic choice $\kappa_1 = \tfrac{1}{40}$.
The above mesh is smooth in the sense that there exist constants $C_0, C_1, C_2 >0$ (independent of $i$, $m$) such that the mesh widths 
$\Delta y_{k,i} = y_{k,i} - y_{k,i-1}$ satisfy
\begin{equation*}\label{smooth}
C_0\, \Delta \xi \le \Delta y_{k,i}  \le C_1\, \Delta \xi ~~{\rm and}~~  |\Delta y_{k,i+1}  - \Delta y_{k,i}| \le C_2 \left( \Delta \xi \right)^2.
\end{equation*}

The spatial derivatives in \eqref{eq:PDE_y} are discretized using central finite difference schemes. 
Let $f: \R \rightarrow \R$ be any given smooth function, let $\cdots < \eta_{i-1} < \eta_i < \eta_{i+1} < \cdots$ be any given smooth mesh
and denote the mesh widths by $h_i = \eta_i-\eta_{i-1}$. 
Then second-order approximations to the first and second derivatives are given by
\begin{equation*}
\begin{split}
f^\prime (\eta_i) &\approx \beta_{i,-1}\, f(\eta_{i-1}) + \beta_{i,0}\, f(\eta_{i}) + \beta_{i,1}\, f(\eta_{i+1}),\\\\
f^{\prime\prime} (\eta_i) &\approx \gamma_{i,-1}\, f(\eta_{i-1}) + \gamma_{i,0}\, f(\eta_{i}) + \gamma_{i,1}\, f(\eta_{i+1}),
\end{split}
\end{equation*}
with
\begin{equation*}
\beta_{i,-1} = \frac{-h_{i+1}}{h_{i}(h_{i}+h_{i+1})}~,~
\beta_{i,0} = \frac{h_{i+1}-h_{i}}{h_{i}h_{i+1}}~,~
\beta_{i,1} = \frac{h_{i}}{h_{i+1}(h_{i}+h_{i+1})}~,
\end{equation*}
and
\begin{equation*}
\gamma_{i,-1} = \frac{2}{h_{i}(h_{i}+h_{i+1})}~,~
\gamma_{i,0} = \frac{-2}{h_{i}h_{i+1}}~,~
\gamma_{i,1} = \frac{2}{h_{i+1}(h_{i}+h_{i+1})}~.
\end{equation*}
The above two finite difference formulas are applied with $\eta_i = y_{k,i}$ for $1\le i\le m$ and $1\le k\le d$.

Semidiscretization of the PDE for $w^{(1,\, l)}$ on the plane segment $P_l$ leads to a system of ordinary differential equations (ODEs) 
\begin{equation}\label{eq:ODE_y}
W^\prime(t) =
\left( \lambda_1 A_1 +  \lambda_l A_l \right) W(t)
\end{equation}
for $t\in (\alpha_{e-1},\alpha_e)$, $1\le e \le E$.
Here $W(t)$ is a vector of dimension $m^2$ and $A_1$, $A_l$ are given $m^2 \times m^2$ matrices that are tridiagonal (possibly up 
to permutation) and correspond to, respectively, the first and the $l$-th spatial direction.
The ODE system \eqref{eq:ODE_y} is completed by an initial condition 
\[
W(0) = \Psi_0
\] 
and, for $1\le e \le E-1$, an optimal exercise 
condition
\begin{equation*}
W(\alpha_e) = \max \left( \Psi_{e} \,,\, \lim_{t\uparrow \alpha_e} W(t) \right).
\end{equation*}
Here the vector $\Psi_e$ is determined by the function $\psi(\cdot,\alpha_e)$ on $P_l$ for $0\le e \le E-1$.
The maximum of any given two vectors is to be taken componentwise.

The payoff function $\phi$ given by \eqref{eq:payoff} is continuous but not everywhere differentiable, and hence, this also holds 
for the function $\psi$ given by \eqref{eq:psi}.
It is well-known that the nonsmoothness of the payoff function can have an adverse impact on the convergence behaviour of the spatial 
discretization.
To alleviate this, we employ cell averaging near the points of nonsmoothness in defining the initial vector $\Psi_0$, see e.g.~\cite{H17}.

For the temporal discretization of the ODE system \eqref{eq:ODE_y} a standard Alternating Direction Implicit (ADI) method is applied.
Consider a given step size $\Delta t = T/N$ with integer $N\ge E$ and define temporal grid points $t_n = n \Delta t$ for $n=0,1,\ldots,N$.
Assume that $\alpha_{e} = t_{n_e}$ for some integer $n_e$ whenever $e=1,2,\ldots,E-1$. Let $W_0 = \Psi_0$ and
\[
{\cal N} = \{n_1,n_2,\ldots,n_{E-1}\}.
\]
Application of the familiar second-order Brian and Douglas ADI scheme for two-dimensional PDEs leads to an approximation 
$W_n\approx W(t_n)$ that is successively defined for $n=1,2,\ldots,N$ by
\begin{equation}\label{eq:Douglas}
\left\{
\begin{array}{l}
Z_0 = W_{n-1}+\Delta t \left( \lambda_1 A_1 +  \lambda_l A_l \right) W_{n-1}, \\\\
Z_1 = Z_0+\tfrac{1}{2}\Delta t\, \lambda_1 A_1 (Z_1-W_{n-1}), \\\\
Z_2 = Z_1+\tfrac{1}{2}\Delta t\, \lambda_l A_l\, (Z_2-W_{n-1}), \\\\
W_n = Z_2~~({\rm if}~n \not\in {\cal N}) \quad {\rm and} \quad W_n = \max (\Psi_{e}\,,\,Z_2)~~({\rm if}~n=n_e\in {\cal N}).
\end{array}\right.
\end{equation}
In the scheme \eqref{eq:Douglas} a forward Euler predictor stage is followed by two implicit but unidirectional corrector stages, 
which serve to stabilize the predictor stage.
The two linear systems in each time step can be solved very efficiently by using a priori $LU$ factorizations of the pertinent 
matrices.
As for the spatial discretization, also the convergence behaviour of the temporal discretization can be adversely effected by 
the nonsmooth payoff function.
To remedy this, backward Euler damping (or Rannacher time stepping) is applied at initial time as well as at each exercise date, 
that is, with $n_0=0$, the time step from $t_{n_e}$ to $t_{n_e+1}$, is replaced by two half steps of the backward Euler method 
for $e=0,1,\ldots,E-1$.

Finally, discretization of the PDE for $w^{(1)}$ on the line segment $L_1$ is performed completely analogously to 
the above.
Then a semidiscrete system $W^\prime(t) = \lambda_1 A_1  W(t)$ is obtained with $W(t)$ a vector of dimension $m$ and $A_1$ 
an $m \times m$ tridiagonal matrix.
Temporal discretization is done using the Crank--Nicolson scheme with backward Euler damping.

\section{Numerical experiments}\label{sec:numericalexperiments}
In this section we investigate by ample numerical experiments the error of the discretization described in Section \ref{sec:discretization} 
of the PCA-based approximation ${\widetilde w}(Y_0,T)$ defined in Section \ref{sec:approximation}.
We consider three parameter sets for the basket option and the underlying asset price model.

Set~A is given by Reisinger \& Wittum \cite{RW07} and has $d=5$, $K=1$, $T=1$, $r=0.05$ and
\[
(\rho_{ij})_{i,j=1}^d = 
\begin{pmatrix}
1.00 & 0.79 & 0.82 & 0.91 & 0.84\\
0.79 & 1.00 & 0.73 & 0.80 & 0.76\\
0.82 & 0.73 & 1.00 & 0.77 & 0.72\\
0.91 & 0.80 & 0.77 & 1.00 & 0.90\\
0.84 & 0.76 & 0.72 & 0.90 & 1.00
\end{pmatrix},\qquad~~
\]
\[
(\sigma_i)_{i=1}^d = 
\begin{pmatrix}
0.518 & 0.648 & 0.623 & 0.570 & 0.530
\end{pmatrix},
\]
\[
(\omega_i)_{i=1}^d = 
\begin{pmatrix}
0.381 & 0.065 & 0.057 & 0.270 & 0.227
\end{pmatrix}.
\]
The eigenvalues of the corresponding covariance matrix $\Sigma$ are 
\[
(\lambda_i)_{i=1}^d = \begin{pmatrix} 1.4089 & 0.1124 & 0.1006 & 0.0388 & 0.0213 \end{pmatrix}.
\]
Hence, $\lambda_1$ is clearly dominant.

Sets~B and C are taken from Jain \& Oosterlee \cite{JO15} and have dimensions $d=10$ and $d=15$, 
respectively. Here $K=40$, $T=1$, $r=0.06$ and $\rho_{ij}=0.25$, $\sigma_i = 0.20$ and 
$\omega_i = 1/d$ whenever $1\le i\not= j \le d$.
Sets B and C have $\lambda_1 = 0.13$ and $\lambda_1 = 0.18$, respectively, and $\lambda_2 = \ldots = \lambda_d = 0.03$.
Thus $\lambda_1$ is also dominant for these parameter sets.
It can be shown that for all three Sets~A,~B,~C the matrix of eigenvectors $Q$ of $\Sigma$ satisfies Assumption \ref{MatrixQ}.

We consider a Bermudan basket option with $E=10$ exercise times $\tau_i = i\, T/E$ ($i=1,2,\ldots,E$) and study the absolute 
error in the discretization of ${\widetilde w}(Y_0,T)$ at the point $Y_0 = y(x(S_0,T))$ with $S_0 = (K,K,\ldots,K)^\rT$.
For comparison, also the European basket option is included in the experiments.
The number of time steps is taken as $N=m$ for the European option and $N = 2E \lceil m/E \rceil$ for the Bermudan option.

\begin{table}[h!]
\centering
\begin{tabular}{ c | c  c}
 & European & Bermudan \\
 \hline
Set~A & 0.17577 & 0.18041\\
Set~B & 0.83257 & 1.05537\\
Set~C & 0.77065 & 0.99277
\end{tabular}
\caption{Reference values for ${\widetilde w}(Y_0,T)$.}\label{tabrefs1}
\end{table}
\hspace{0.5cm}\\

Table~\ref{tabrefs1} provides reference values for ${\widetilde w}(Y_0,T)$, which have been computed by choosing $m=1000$.
In the case of Set~A, Reisinger \& Wittum \cite{RW07} obtain the approximation $w(Y_0,T) \approx 0.1759$ for the European basket option.
In the case of Sets~B and C, Jain \& Oosterlee \cite{JO15} obtain, using the stochastic grid bundling method, 
the approximations $w(Y_0,T) \approx 1.06$ 
and $w(Y_0,T) \approx 1.00$, respectively, for the Bermudan basket option.
Clearly, these three approximations from the literature agree well with our corresponding values for ${\widetilde w}(Y_0,T)$ given in 
Table~\ref{tabrefs1}.

Figure~\ref{fig1} displays the absolute error in the discretization of ${\widetilde w}(Y_0,T)$ versus $1/m$ for all $m=10,11,12,\ldots,100$.
Here both the European and Bermudan basket options are considered for all three parameter sets A, B, C.
The favourable result is observed that the discretization error is always bounded from above by $c m^{-2}$ with a moderate constant $c$, 
which is as desired. 

For the European option and Set~A, the error drop in the (less important) region $m\le 20$ is somewhat surprising, but it is easily 
explained from a change of sign in the error.
Except for this, in the case of the European basket option, the error behaviour is always found to be regular and second-order.

For the Bermudan basket option the observed error behaviour is less regular, in particular in the interesting region of large values $m$.
To gain more insight into this phenomenon, we have computed separately the discretization error for the leading term $w^{(1)}(Y_0,T)$ and 
for the correction term $\sum_{l=2}^d \left[ w^{(1,\, l)}(Y_0,T) - w^{(1)}(Y_0,T) \right]$ in ${\widetilde w}(Y_0,T)$, 
see \eqref{eq:solutionApprox2}.
Reference values for the leading term are given in Table~\ref{tabrefs2}.

\begin{table}[h!]
\centering
\begin{tabular}{ c | c  c}
 & European & Bermudan \\
 \hline
Set~A & 0.18061 & 0.18407\\
Set~B & 1.00043 & 1.17792\\
Set~C & 0.94368 & 1.11902
\end{tabular}
\caption{Reference values for $w^{(1)}(Y_0,T)$.}\label{tabrefs2}
\end{table}
\hspace{0.5cm}\\
\noindent
The obtained result is shown in Figure~\ref{fig2}, where dark squares indicate the error $e^{(1)}(m)$ for the leading term and light circles 
the error $\sum_{l=2}^d \left[ e^{(1,\, l)}(m) - e^{(1)}(m) \right]$ for the correction term.
It is clear that in all six cases the error for the leading term behaves regularly and the error for the correction term is small compared to 
this (for Set~A if $m\ge 20$, as above).
For the Bermudan basket option, however, the behaviour of the discretization error for the correction term is rather irregular.
A subsequent study shows that for any given $l$ the error $e^{(1,\, l)}(m)$ is always very close to the error $e^{(1)}(m)$, which is as 
expected, but the difference can be both positive and negative, leading to an irregular behaviour of $e^{(1,\, l)}(m) - e^{(1)}(m)$. 
This is exacerbated when summing these differences up over $l=2,3,\ldots,d$. Hence, the irregular behaviour of the error for the correction 
term can adversely affect the regular behaviour of the error for the leading term.
We remark that this has been observed in many other experiments we performed for the Bermudan basket option, for example for other points $Y_0$, 
for other numbers of exercise times $E\ge 2$, for other dimensions $d\ge 3$ and for other covariance matrices $\Sigma$, having 
$\lambda_1 \gg \lambda_2 > \cdots > \lambda_d > 0$.
It is our aim of future research to find a remedy for this phenomenon.

\section{Conclusions}\label{sec:conclusion}
In this paper we have investigated the PCA-based approach by Reisinger \& Wittum \cite{RW07} for the valuation of Bermudan basket options.
This approximation approach is highly effective as it requires the solution of only a limited number of low-dimensional PDEs, supplemented
with optimal exercise conditions.
By numerical experiments the favourable result is shown that a common discretization of these PDE problems leads to a second-order 
convergence behaviour in space and time.
It is also observed that this convergence behaviour can be somewhat irregular.
Insight into this phenomenon is obtained by regarding the total discretization error as a superposition of discretization errors for the 
leading term and the correction term.
Our aim for future research is to determine a suitable remedy for it.
Another topic for future research concerns a rigorous analysis of the error in the PCA-based approximation for Bermudan 
basket options.
The results obtained by Reisinger \& Wissmann \cite{RW17}, relevant to European basket options, will be important here.

\bibliographystyle{plain}
\bibliography{references}
 
\clearpage 
\begin{figure}[h!]
    \centering
    \hspace{-0.5cm}\includegraphics[trim={0cm 0cm 1cm 0cm},clip,
    width =0.51\textwidth]{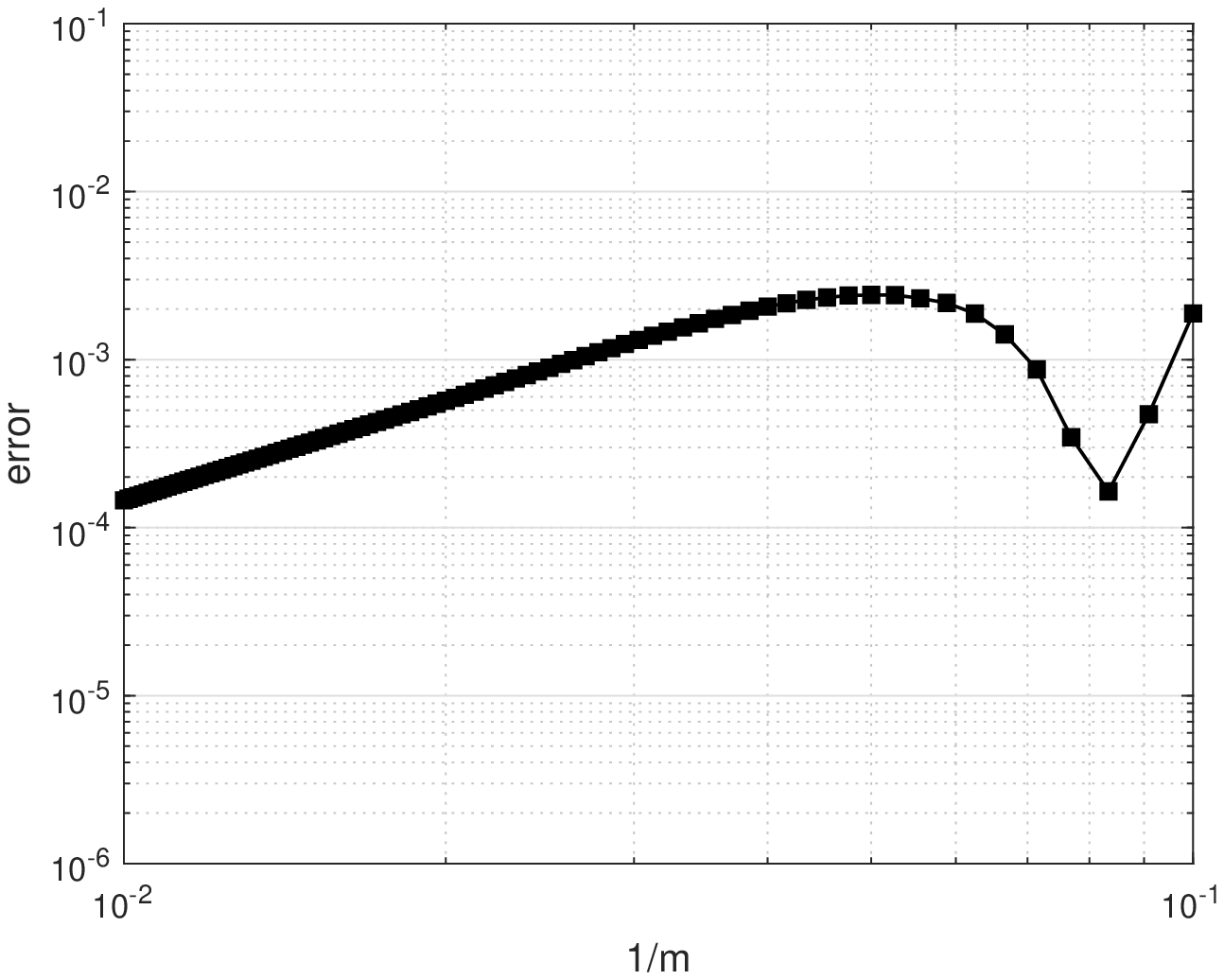}
    \includegraphics[trim={0cm 0cm 1cm 0cm},clip,
    width =0.51\textwidth]{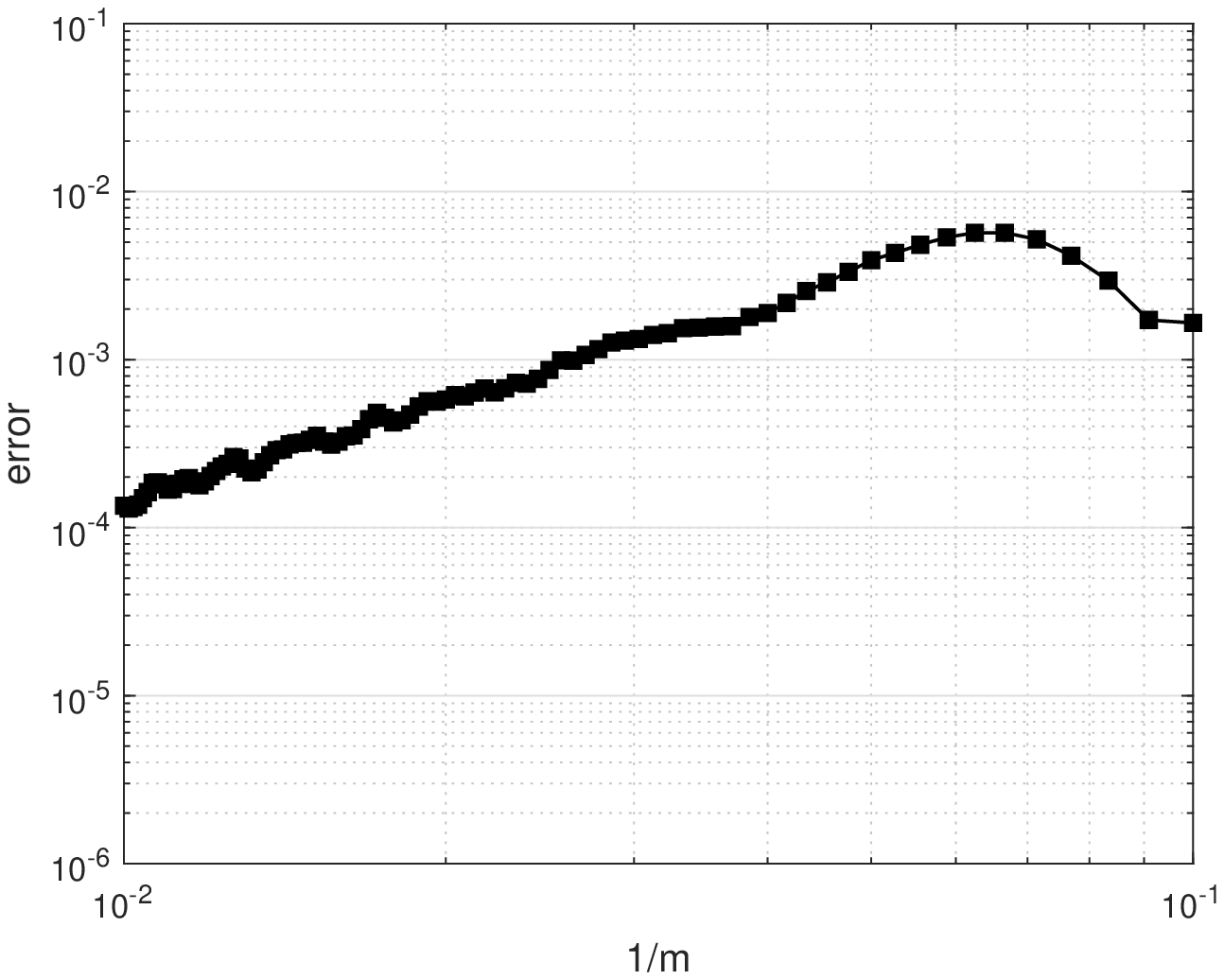}\\
    \hspace{-0.5cm}\includegraphics[trim={0cm 0cm 1cm 0cm},clip,
    width =0.51\textwidth]{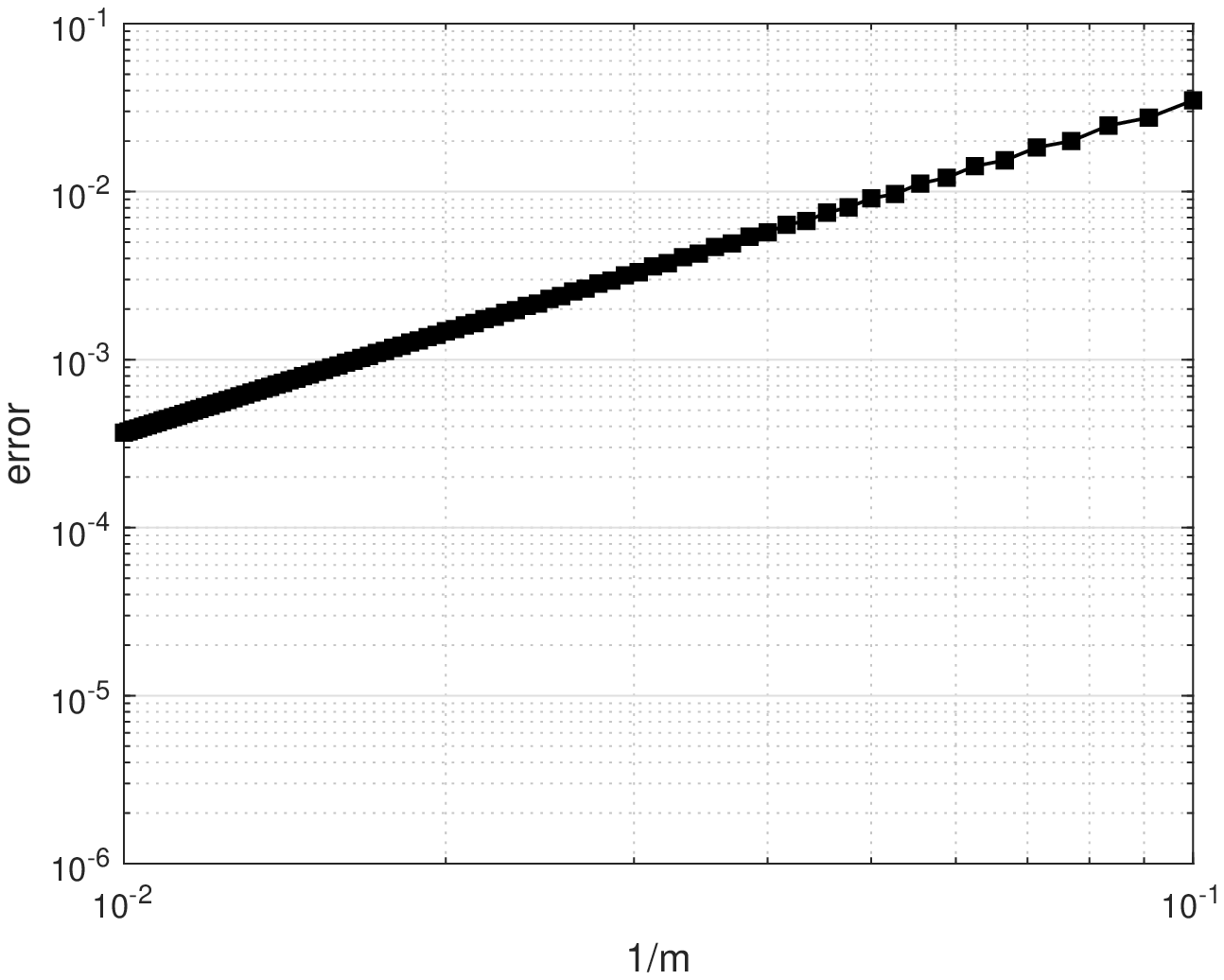}
    \includegraphics[trim={0cm 0cm 1cm 0cm},clip,
    width =0.51\textwidth]{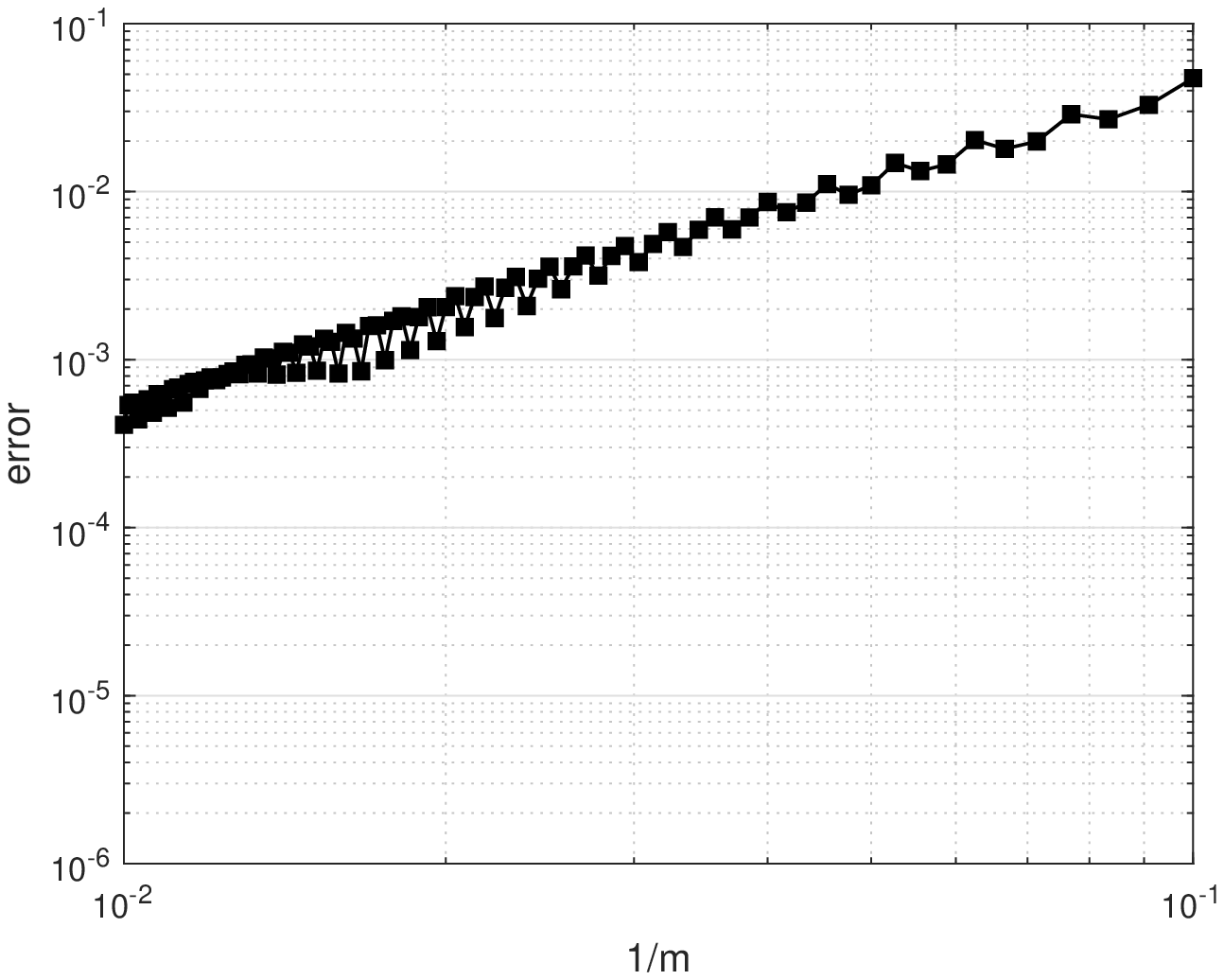}\\
    \hspace{-0.5cm}\includegraphics[trim={0cm 0cm 1cm 0cm},clip,
    width =0.51\textwidth]{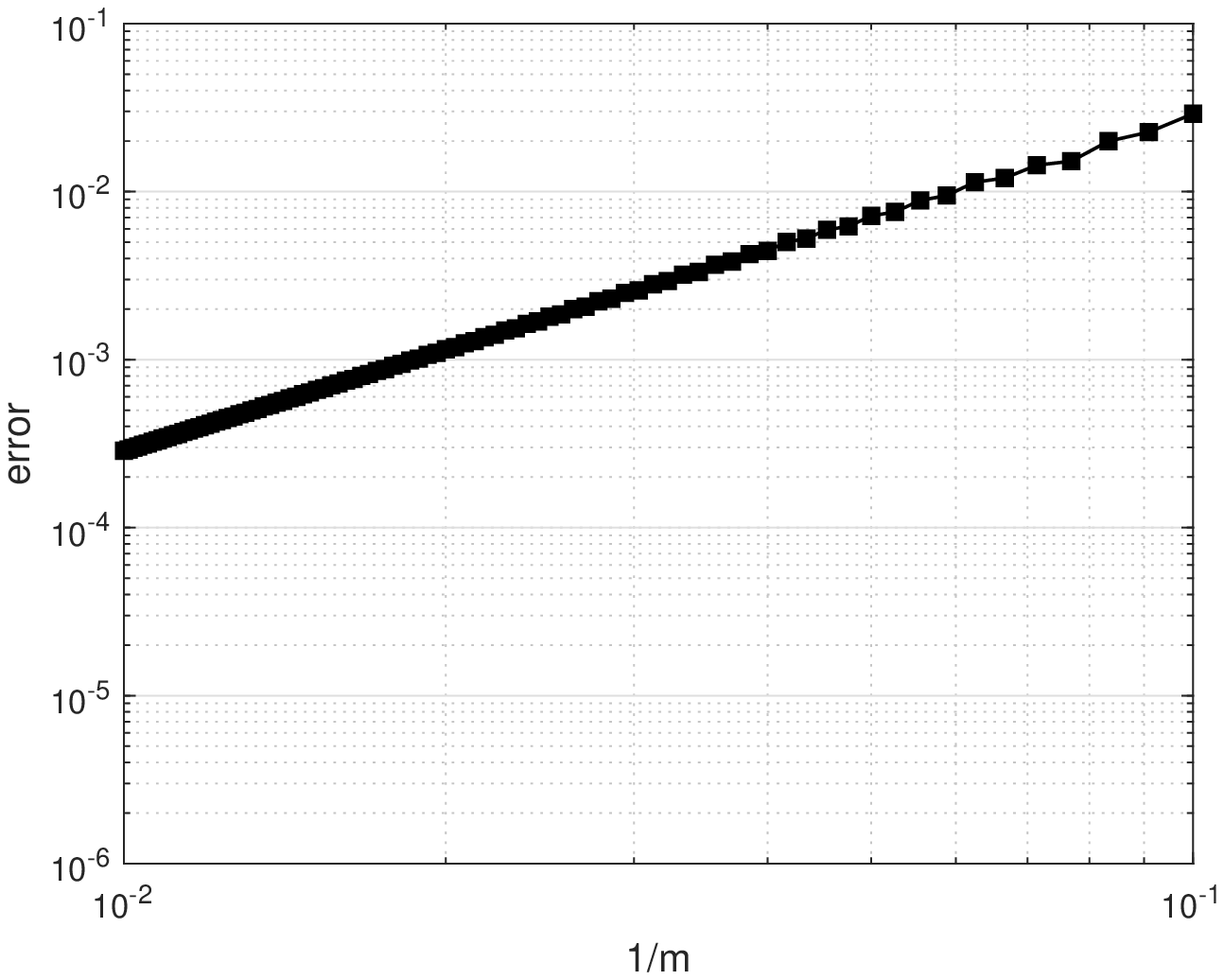}
    \includegraphics[trim={0cm 0cm 1cm 0cm},clip,
    width =0.51\textwidth]{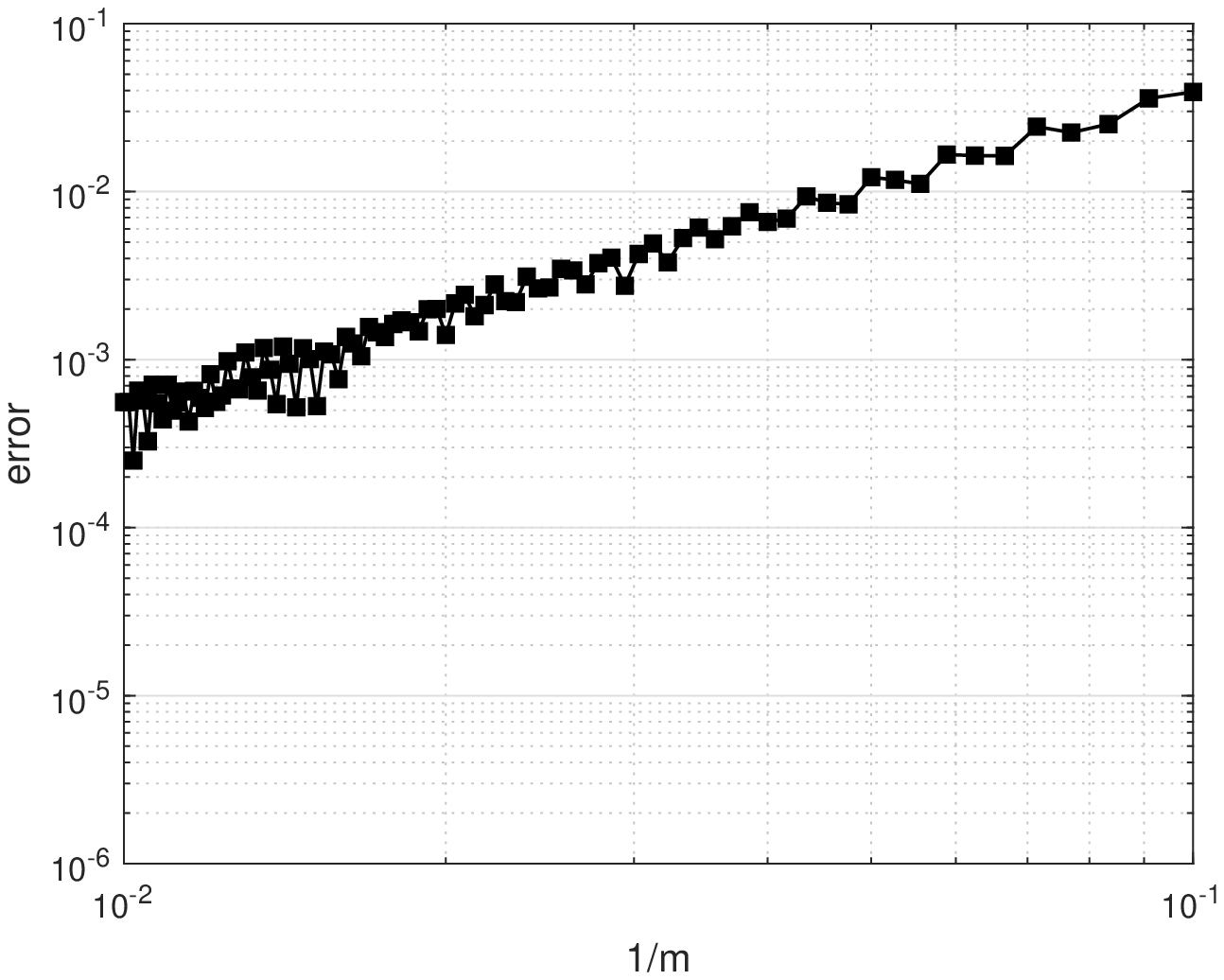}\\
    \caption{Discretization error for ${\widetilde w}(Y_0,T)$ in all cases of Table~\ref{tabrefs1}.
    Left: European basket option. Right: Bermudan basket option. 
    Top: Set~A. Middle: Set~B. Bottom: Set~C}
    \label{fig1}
\end{figure}
\vfill\clearpage

\begin{figure}[h!]
    \centering
    \hspace{-0.5cm}\includegraphics[trim={0cm 0cm 1cm 0cm},clip,
    width =0.51\textwidth]{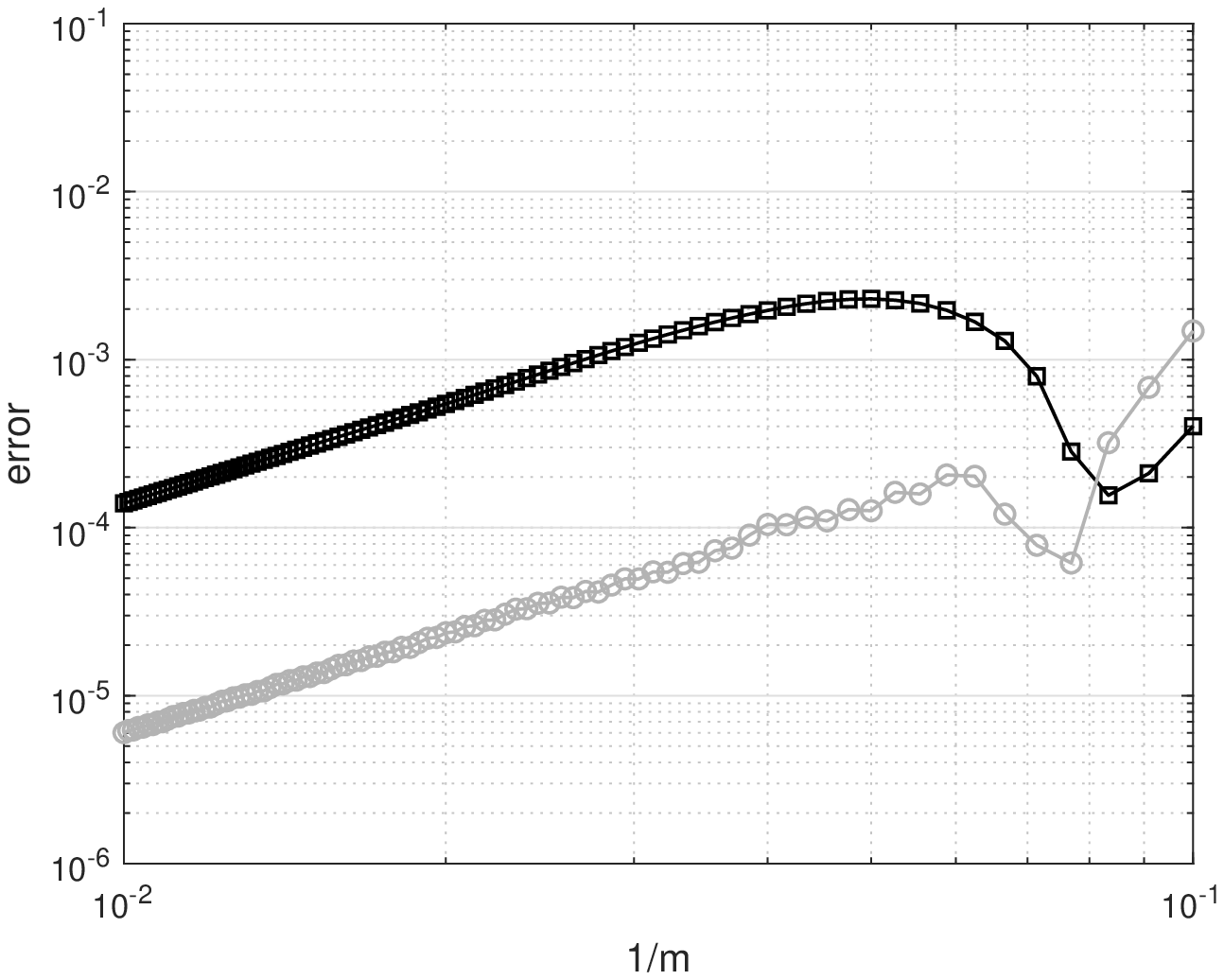}
    \includegraphics[trim={0cm 0cm 1cm 0cm},clip,
    width =0.51\textwidth]{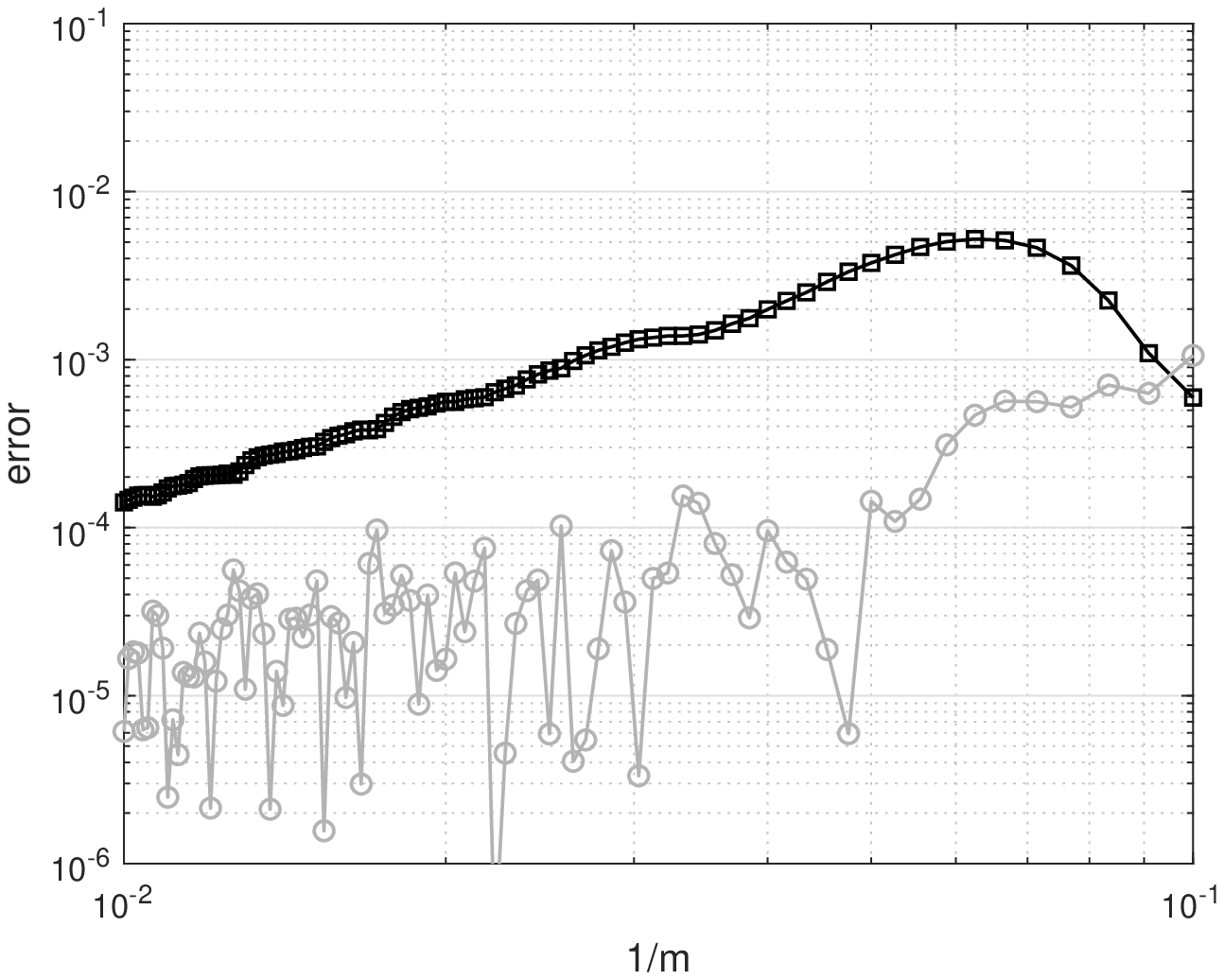}\\
    \hspace{-0.5cm}\includegraphics[trim={0cm 0cm 1cm 0cm},clip,
    width =0.51\textwidth]{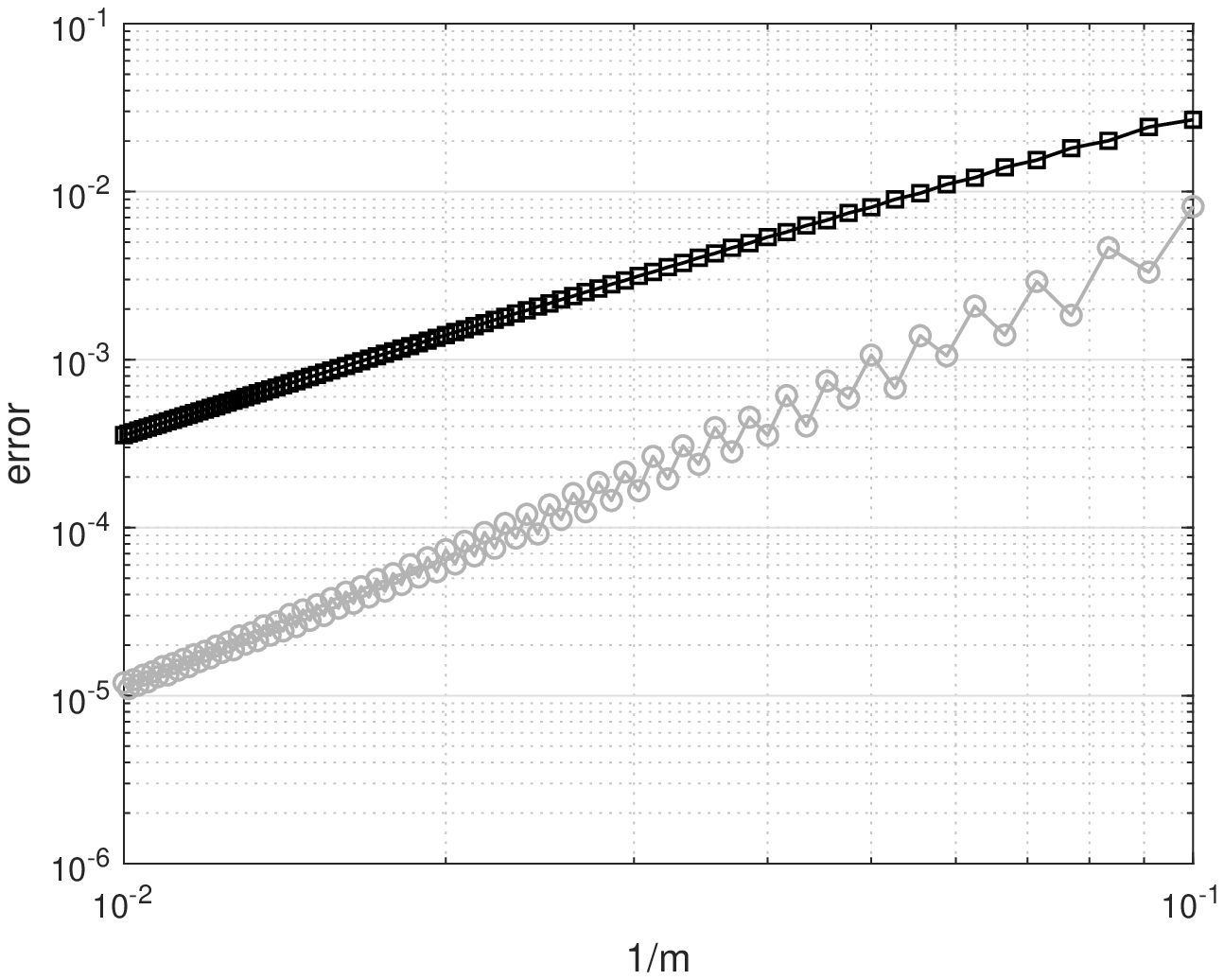}
    \includegraphics[trim={0cm 0cm 1cm 0cm},clip,
    width =0.51\textwidth]{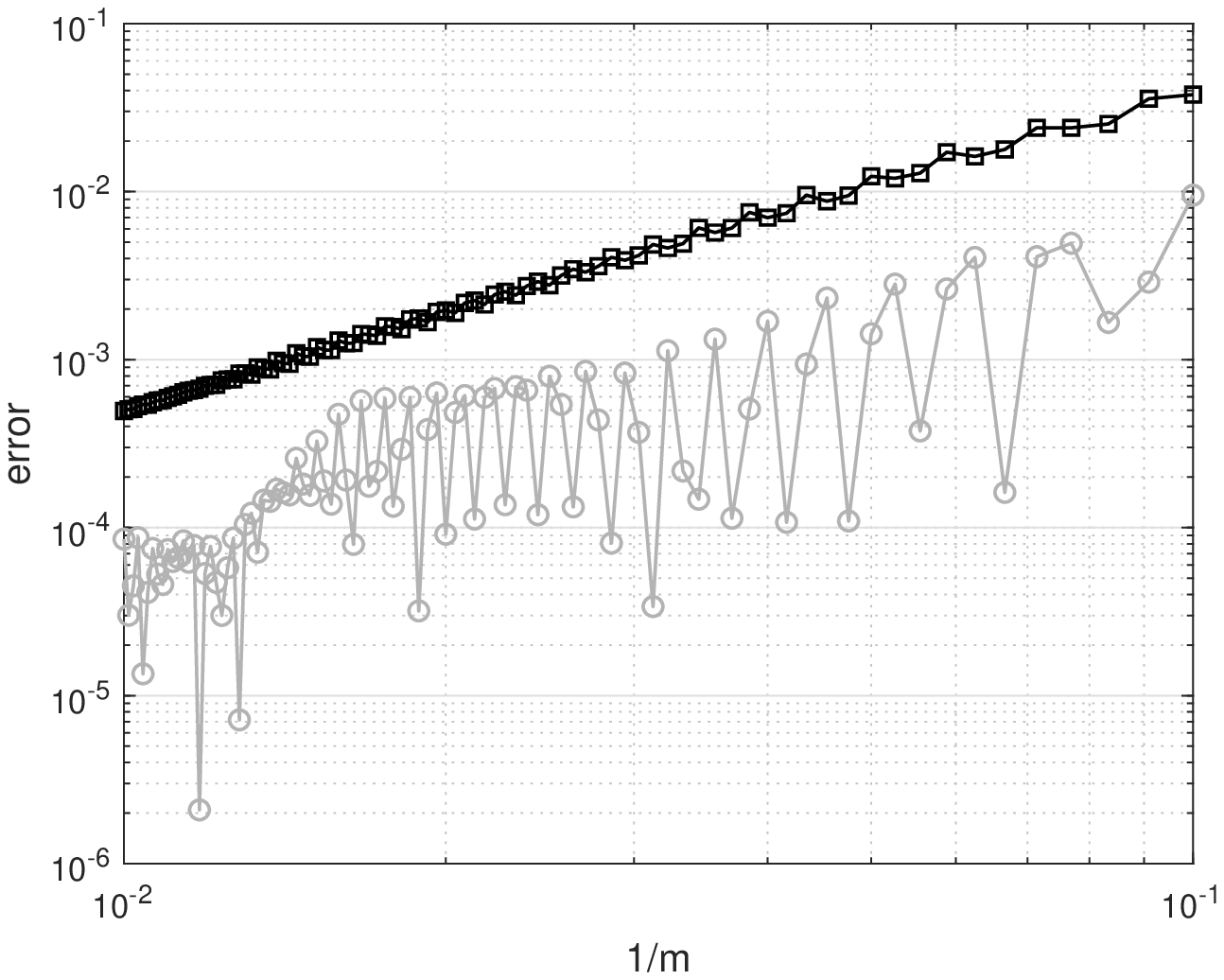}\\
    \hspace{-0.5cm}\includegraphics[trim={0cm 0cm 1cm 0cm},clip,
    width =0.51\textwidth]{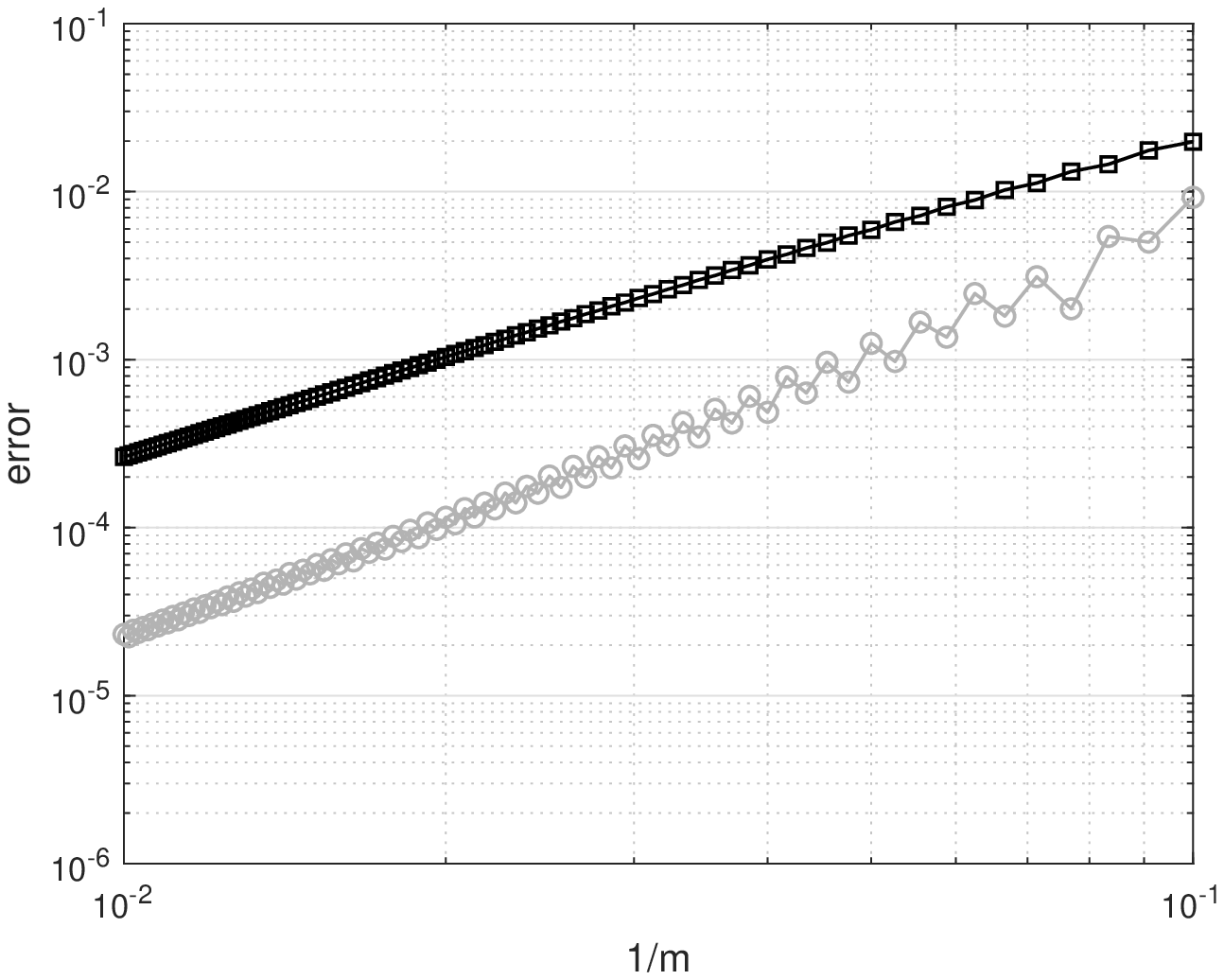}
    \includegraphics[trim={0cm 0cm 1cm 0cm},clip,
    width =0.51\textwidth]{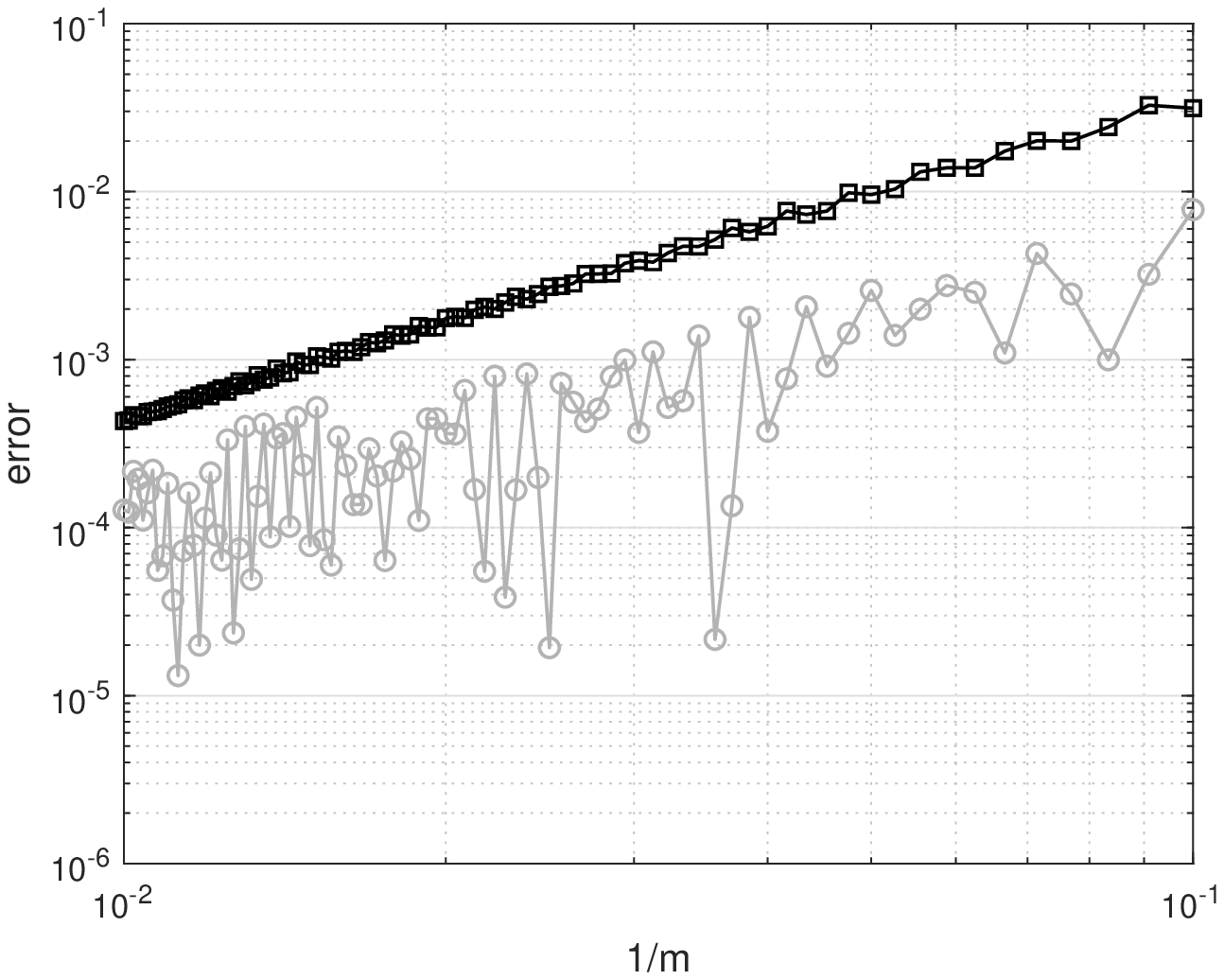}\\
    \caption{Discretization error for the leading term (dark squares) 
    and the correction term (light circles) in ${\widetilde w}(Y_0,T)$.
    Left: European basket option. Right: Bermudan basket option. 
    Top: Set~A. Middle: Set~B. Bottom: Set~C}    \label{fig2}
\end{figure}
\vfill\clearpage

\appendix

\section{Dirichlet boundary condition for \eqref{eq:PDE_y}}\label{sec:BC}
\mbox{Consider the following minor assumption on the matrix~$Q$ of eigenvectors of the covariance matrix $\Sigma$.}
\begin{assumption}\label{MatrixQ}
Each column of $Q$ satisfies one of the following two conditions: 
\begin{itemize}
\item[(a)] all its entries are strictly positive;
\item[(b)] it has both a strictly positive and a strictly negative entry.
\end{itemize}
\end{assumption}
Then we have
\begin{lemma}\label{psilimit}
Let the function $\psi$ be given by \eqref{eq:psi} with $\phi$ defined by \eqref{eq:payoff}. 
Let $k \in \{1,2,\ldots,d\}$, $t\in [0,T]$ and $y = (y_1,y_2,\ldots,y_d)^\rT$ with fixed $y_j \in (0,1)$ whenever 
$j \not= k$. 
If the $k$-th column of $Q$ satisfies (A.1.a), then $\psi (y,t) \rightarrow K$ as $y_k \downarrow 0$.
If the $k$-th column of $Q$ satisfies (A.1.b), then $\psi (y,t) \rightarrow 0$ as $y_k \downarrow 0$.
Finally, $\psi (y,t) \rightarrow 0$ as $y_k \uparrow 1$.
\end{lemma}
\begin{proof}
Let $x = \tan \left[ \pi(y-\tfrac{1}{2}) \right]$ and $s =  K \exp \left[ Qx +b(t) \right]$, so that $\psi (y,t) = \phi(s)$.

Suppose first $y_k \downarrow 0$. Then $x_k \rightarrow -\infty$. 
If the $k$-th column of $Q$ satisfies condition (A.1.a), then all entries of $Qx$ tend to $-\infty$.
Consequently, all entries of $s$ tend to zero and $\phi(s)\rightarrow K$.
If the $k$-th column of $Q$ satisfies condition (A.1.b), then the entries of $Qx$ go to either $-\infty$ or $+\infty$ with
at least one entry  that tends to $+\infty$.
It follows that the entries of $s$ go to either zero or $+\infty$ with at least one entry that tends to $+\infty$,
and therefore $\phi(s)\rightarrow 0$.

Suppose next $y_k \uparrow 1$. Then $x_k \rightarrow +\infty$ and the entries of $Qx$ go to either $+\infty$ or $-\infty$ 
with at least one entry that tends to $+\infty$.
Hence, $\phi(s)\rightarrow 0$.
\mbox{\tiny $\blacksquare$}
\end{proof}
\\

For any given $k \in \{1,2,\ldots,d\}$ the diffusion and convection coefficients $p(y_k)$ and $q(y_k)$ in \eqref{eq:PDE_y} 
vanish as $y_k \downarrow 0$ or $y_k \uparrow 1$.
Accordingly, \eqref{eq:PDE_y} is also satisfied on each boundary part 
\begin{equation*}
\{ y: y = (y_1,y_2,\ldots,y_d)^\rT  ~{\rm with}~ y_k=\delta ~{\rm and}~ y_j \in (0,1) ~{\rm whenever}~ j \not= k\}
\end{equation*}
for $\delta\in \{0,1\}$.
Also the initial condition \eqref{eq:IC_y} and optimal exercise condition \eqref{eq:ICe_y} hold on each such boundary 
part, upon taking the relevant limit value for $\psi (y,t)$ given by Lemma \ref{psilimit}.
On each part where this limit value equals $K$, the solution \eqref{eq:BC_y} is obtained, and on each part where the limit 
value equals zero, the zero solution holds.
This yields the Dirichlet boundary condition for the PDE \eqref{eq:PDE_y} stated in Section \ref{sec:transformation}.

\end{document}